\documentclass[12pt, reqno]{amsart}
\usepackage{graphicx}
\usepackage{hyperref}
\usepackage[caption = false]{subfig}
\usepackage{amsthm}
\usepackage{amssymb}
\usepackage{mathrsfs}
\usepackage{mathtools}
\usepackage{enumerate}
\usepackage{amssymb}
\usepackage{wrapfig}
\usepackage{float}
\allowdisplaybreaks

\usepackage[twoside,paperwidth=190mm, paperheight=297mm, top=35mm, bottom=35mm, left=30mm, right=26mm, marginparsep=3mm, marginparwidth=40mm]{geometry}

\numberwithin{equation}{section}

\theoremstyle{plain}

\newtheorem{theorem}{Theorem}[section]
\newtheorem{corollary}[theorem]{Corollary}

\theoremstyle{definition}
\newtheorem{definition}[theorem]{Definition}
\theoremstyle{remark}

\makeatother

\begin{document}

\title{Majorization for Starlike and Convex functions with respect to Conjugate points}
	\thanks{K. Gangania thanks to University Grant Commission, New-Delhi, India for providing Junior Research Fellowship under UGC-Ref. No.:1051/(CSIR-UGC NET JUNE 2017).}

	\author[Kamaljeet]{Kamaljeet Gangania}
	\address{Department of Applied Mathematics, Delhi Technological University,
		Delhi--110042, India}
	\email{gangania.m1991@gmail.com}
	
	\author[S. Sivaprasad Kumar]{S. Sivaprasad Kumar}
	\address{Department of Applied Mathematics, Delhi Technological University,
		Delhi--110042, India}
	\email{spkumar@dce.ac.in}

\maketitle	
	
\begin{abstract} 
	  The concept of majorization is now  well-known after the beautiful work of MacGregor, and then followed by Campbell in his sequel of papers. In this paper, we establish the sharp majorization results for the starlike and convex functions with respect to the conjugate points.	
\end{abstract}
\vspace{0.5cm}
	\noindent \textit{2010 AMS Subject Classification}. Primary 30C45, 30C50, Secondary 30C80.\\
	\noindent \textit{Keywords and Phrases}. Starlike and Convex function, Radius problems, Majorization, Conjugate points, Ma-Minda classes.

\maketitle
	
\section{Introduction}
Let $\mathcal{A}$ be the set of all normalized analytic functions $f(z)=z+ \sum_{n=2}^{\infty}a_n z^n$ in the open unit disc $\mathbb{D}:= \{ z: \vert z \vert <1\}$. The subclass of $\mathcal{A}$ of univalent functions be denoted by $\mathcal{S}$. Recall the following definition in want of our onward results:
\begin{definition}\cite{sub-pg132}
	Let $f$, $g$ and $\omega$ be analytic in $\vert z\vert <r$. 
	The function $g$ is majorized by $f$ denoted by $g<< f$  in $\vert z \vert <r$, if $\vert g(z) \vert \leq \vert f(z) \vert $ in $\vert z \vert<r$. 
	Let us recall that $g$ is subordinate to $f$ denoted by $g \prec f$ in $\vert z \vert < r$ if $g(z) = f(\omega(z))$, where $\vert \omega(z) \vert \leq \vert z \vert$ and  $\omega(0)$ in $\vert z \vert < r$. Further, if $f$ is univalent then $g\prec f$ if and only if $g(\mathbb{D}_r) \subseteq f(\mathbb{D}_r)$, where $\mathbb{D}_r := \{z : \vert z \vert <r \}$.
\end{definition}

In $1936$, Biernacki~\cite{Biernacki-1936} introduced Majorization-Subordination theory, proving that if  $g'(0) \geq 0$ and $g \prec f$, where $ f\in \mathcal{S}$ in $\mathbb{D}$, then $g << f$ in $\vert z \vert<1/4$. In the subsequent years, Goluzin, Tao Shah, Lewandowski and MacGregor examined a variety of related problems (for greater detail see~\cite{Campbell-1972}). In 1951, Goluzin~\cite{Goluzin-1969} proved that if $g'(0) \geq 0$ and $g \prec f$, where $f\in \mathcal{S}$ then $g' << f'$ in $\vert z \vert< 0.12$ and hypothesized the majorization radius as $\vert z \vert < 3 - \sqrt{8}$. Later, Tao Shah found this to be true in $1958$.

In $1967$, MacGregor~\cite{mc} proved sharp majorization for the class of univalent starlike and convex functions. Later, Campbell~\cite{Campbell-1972,campbell-1973,campbell-1974} obtained sharp majorization results for locally univalent functions in his sequel of three papers. Since then many authors proved majorization results for classes of meromorphic functions~\cite{Goyal,TangAouf-2015}, but claim for the sharpness is still open. In $2019$, Teng and Deng~\cite{tang} obtained results for several subclasses of starlike functions defined in view of Ma and Minda~\cite{minda94} classes, but here also sharpness was missing and it was not possible to mimic the method of sharpness used by MacGregor~\cite{mc}. This made the problem more interesting for the general Ma-Minda classes, which in 2021 solved by us~\cite{ganga_cmft2021}. 

Let us now recall an equivalent definition of majorization from \cite{mc}.
\begin{definition}\label{mc-def}\cite{mc}
	Let $f$ and $g$ be analytic in $\mathbb{D}$. A function  $g(z)$ is said to be majorized by $f(z)$, denoted by $g<<f$, if there exists an analytic function $\Phi(z)$ in $\mathbb{D}$ satisfying
	$ \vert \Phi(z)\vert \leq1$ and $g(z)=\Phi(z) f(z)$ for all $z\in\mathbb{D}.$
\end{definition}

\begin{theorem}[MacGregor Theorem~\cite{mc}]\label{mc-theoremunivalent}
	Let $g$ be majorized by $f$ in $\mathbb{D}$ and $g(0)=0$. If $f(z)$ is univalent in $\mathbb{D}$, then $\vert g'(z)\vert \leq \vert f'(z)\vert $ in $\vert z \vert \leq 2-\sqrt{3}$. The constant $2-\sqrt{3}$ is sharp.
\end{theorem}

In $2004$, the class of starlike and convex functions with respect to conjugate points, respectively were unified by Ravichandran~\cite{RaviConjugate-2004}, which are as follows:
\begin{definition}\cite{RaviConjugate-2004}
	Let us consider
	\begin{equation*}
	\mathcal{S}^*_{c}(\psi)= \left\{f\in \mathcal{A}: \frac{2zf'(z)}{f(z)+ \overline{f(\bar{z})} } \prec \psi(z)  \right\}
	\end{equation*}
	and 
	\begin{equation*}
	\mathcal{C}_{c}(\psi)= \left\{f\in \mathcal{A}: \frac{2(zf'(z))'}{(f(z)+ \overline{f(\bar{z})})' } \prec \psi(z)  \right\}.
	\end{equation*}
\end{definition}

For the standard notations and basic results of Ma and Minda classes of starlike and convex functions $\mathcal{S}^*(\psi)$ and $\mathcal{C}(\psi)$ respectively, see\cite{minda94}. For its connection to the classes $\mathcal{S}^*_{c}(\psi)$ and $\mathcal{C}_{c}(\psi)$, see~\cite{RaviConjugate-2004}. For some recent articles, we refer to see~\cite{goel2020,ganga-cardioid} and the references therein.

In this article, we prove sharp results, similar to Theorem~\ref{mc-theoremunivalent} in context of $\mathcal{S}^*_{c}(\psi)$ and $\mathcal{C}_{c}(\psi)$, where technique of sharpness is different.

\section{Majorization}

Let us consider the function $k_{\phi}\in \mathcal{A}$ be given by
\begin{equation*}
1+\frac{zk''_{\phi}(z)}{k'_{\phi}(z)}=\phi(z),
\end{equation*}
with $k_{\phi}(0)=k'_{\phi}(0)-1=0$, where $\phi$ is a Ma-Minda function, see~\cite{minda94}. Now we prove our main result.
\begin{theorem}\label{majorization-conjuagte}
	Let $\phi$ be convex in $\mathbb{D}$ with $\phi(0)=1$ and $\Re\phi(z)>0$.
	Suppose $\psi$ be the function satisfying 
	\begin{equation}\label{briot-sol}
	\psi(z)+\frac{z\psi'(z)}{\psi(z)}=\phi(z).
	\end{equation}
	Further let $m(r):=\underset {\vert z \vert =r}{\min} \vert\psi(z) \vert$. Let $g\in \mathcal{A}$. If $g<<f$ in $\mathbb{D}$, where $f\in \mathcal{C}_{c}(\phi)$ then
	$$\vert g'(z) \vert \leq \vert f'(z) \vert$$ holds in $\vert z \vert \leq r_{\psi}$,
	where $r_{\psi} \in (0,1)$ is the smallest root of
	\begin{equation}\label{r0}
	(1-r^2)m(r)-2r=0.
	\end{equation}
	The radius constant $r_{\psi}$ is sharp when $m(r)=\psi(-r)$.
\end{theorem}
\begin{proof}
	Let $g$ be majorized by $f$, where $f\in \mathcal{C}_{c}(\phi)$. Then from Defintion~\ref{mc-def}, we have
	\begin{equation*} \label{def_p1}
	g(z)= \Phi(z) f(z),
	\end{equation*}
	where $\Phi$ is analytic in $\mathbb{D}$ and satisfy $\vert \Phi(z) \vert \leq 1$. Thus,
	\begin{equation}\label{def-p2}
	g'(z)= \Phi(z) f'(z) + \Phi'(z) f(z).
	\end{equation}
	The well-known Schwarz-Pick inequality for the function $\Phi$ yields
	$$\left \vert \Phi'(z) \right \vert \leq \frac{1-\vert \Phi(z)\vert ^2}{1-\vert z \vert ^2}.$$
	Let us write $\vert \Phi(z) \vert := \beta$.
	Now using growth and distortion theorems~\cite{RaviConjugate-2004} for the class $\mathcal{C}_{c}(\phi)$ in \eqref{def-p2} along with Schwarz-Pick inequality, we get
	\begin{align}\label{def_p3}
	\vert g'(z) \vert &\leq \vert \Phi(z) \vert \vert f'(z) \vert + \vert \Phi'(z) \vert \vert f(z) \vert \nonumber\\
	&\leq \beta k'_{\phi}(r) +\frac{1-\beta^2}{1-r^2} k_{\phi}(r)
	\end{align}
	for $\vert z \vert =r$. Since $\vert f'(z) \vert \leq  k'_{\phi}(r)$. Therefore, 
	\begin{equation}\label{def_p4}
	\left\vert \frac{g'(z)}{f'(z)} \right\vert \leq \beta +\frac{1-\beta^2}{1-r^2} \frac{k_{\phi}(r)}{k'_{\phi}(r)}.
	\end{equation}
	Now let $p$ be the Carath\'{e}odory function and satisfy
	\begin{equation}\label{briot}
	p(z)+\frac{zp'(z)}{p(z)} \prec \phi(z).
	\end{equation}
	Since $\Re\phi(z)>0$ and $\phi$ is convex in $\mathbb{D}$, \cite[Theorem~3.2d, p.~86]{sub-pg132} implies that the solution $\psi$ of \eqref{briot-sol}  exists and
	is analytic in $\mathbb{D}$ with $\Re\psi(z)>0$ and given by:
	\begin{equation*}
	\psi(z):= q(z)\left(\int_{0}^{z}\frac{q(t)}{t}dt\right)^{-1},
	\end{equation*}
	where 
	\begin{equation*}
	q(z)= z\exp\int_{0}^{z}\frac{\phi(t)-1}{t}dt.
	\end{equation*}
	Since $\Re\psi(z)>0$ and $p$ satisfies the subordination~\eqref{briot}, therefore \cite[Lemma~3.2e, p.~89]{sub-pg132} implies that $p\prec \psi$ and $\psi$ is the best univalent dominant.
	Thus, a function $f\in \mathcal{C}(\phi)$ implies $f\in \mathcal{S}^*(\psi)$, where $\psi$ satisfies~\eqref{briot-sol} follows by taking $p(z)=zf'(z)/f(z)$.
	Further, note that a function with real coefficients in $\mathcal{C}(\psi)$ belongs to $\mathcal{C}_{c}(\psi)$. In particular, taking $f=k_{\phi}$ we conclude that
	\begin{equation*}
	\frac{zk'_{\phi}(z)}{k_{\phi}(z)} = \psi(z),
	\end{equation*}
	which gives
	\begin{equation} \label{def_p5}
	\left \vert \frac{k_{\phi}(z)}{k'_{\phi}(z)} \right \vert \leq \frac{r}{m(r)}, \quad \text{for}\quad \vert z \vert =r.
	\end{equation}
	Using the inequality \eqref{def_p5} in \eqref{def_p4} gives
	\begin{equation}\label{def_p6}
	\left\vert \frac{g'(z)}{f'(z)} \right\vert \leq \beta +\frac{1-\beta^2}{1-r^2} \frac{r}{m(r)} =:h(\beta,r).
	\end{equation}
	Finally, to establish $h(\beta,r) \leq 1$, it is enough to see that
	\begin{equation*}
	\frac{\partial}{\partial\beta}h(\beta,r)=-r<0.
	\end{equation*}
	Hence, $h(\beta,r) \leq 1$ holds in $\vert z \vert =r \leq r_{\psi}$, where $r_{\psi}$ is the smallest postive root of 
	\begin{equation*}
	h(1,r):= (1-r^2)m(r)-2r=0. 
	\end{equation*}
	The existence of the root $r_\psi$ is evident.\\
	
	To show the result is sharp, let $m_r=\psi(-r)$ and $\Phi(z)=(z+\delta)/(1+\delta z)$, where $-1\leq \delta \leq 1$. Take $f(z)\in \mathcal{C}(\phi)$ such that $zf'(z)/f(z)=\psi(-z)$. Let us consider $r_2$ as the second consecutive positive root (if any) of equation~\eqref{r0}, otherwise take $r_2=1$. We prove that for each $r\in (r_{\psi}, r_2)$ we can select $\delta$ so that $0<f'(r)<g'(r)$, which means $g'$ can not be majorized by $f'$ in $\vert z \vert > r_{\psi}$. First note that the function $f$ is a rotation of $k_{\phi}$ with real coefficients and belongs to $\mathcal{C}_{c}(\phi)$ such that
	\begin{equation}\label{def_p7}
	\frac{f(r)}{f'(r)}=\frac{k_{\phi}(r)}{k'_{\phi}(r)}=\frac{r}{\psi(-r)}.
	\end{equation}
	Since 
	\begin{equation*}
	g'(r)= k'_{\phi}(r) \left( \frac{r+\delta}{1+\delta r}+ \frac{1-{\delta}^2}{(1+\delta r)^2} \frac{k_{\phi}(r)}{k'_{\phi}(r)} \right) =: f'(r) K(\delta, r)
	\end{equation*}
	and $K(1,r)=1$, it suffices to show that  ${\partial K(\delta,r)}/{\partial \delta}<0$ at $\delta=1$ in order to establish that $K(1-\epsilon,r)>1$, and hence $g'(r)>f'(r)>0$. But at $\delta=1$, we have:
	\begin{align*}
	\frac{\partial K( \delta,r)}{\partial \delta } &= \frac{2}{(1+r)^2} \left( \frac{1-r^2}{2} -\frac{k_{\phi}(r)}{k'_{\phi}(r)} \right)\\
	&= \frac{2}{(1+r)^2} \left( \frac{1-r^2}{2} -\frac{r}{\psi(-r)} \right)\\
	&<0
	\end{align*}
	using \eqref{r0}, \eqref{def_p6}, \eqref{def_p7} and the fact that $h(1,r)<0$ for all $r\in (r_{\psi}, r_2)$.   \qed
\end{proof}

Proof of the following result is omitted as it directly follows from Theorem~\ref{majorization-conjuagte}.	
\begin{theorem}
	Let us assume that
	$$m(r):=\underset{\vert z \vert =r}{\min} \vert \psi(z)\vert=\left\{
	\begin{array}{ll}
	\psi(-r), & \hbox{ if } \psi'(0)>0;  \\
	\psi(r), & \hbox{ if } \psi'(0)<0,
	\end{array}
	\right.$$
	where $\psi$ is a univalent function with positive real part and $\psi(0)=1$.
	Let $g\in \mathcal{A}$. If $g <<f$ in $\mathbb{D}$, where $f\in \mathcal{S}_{c}^{*}(\psi)$ then
	$$\vert g'(z)\vert  \leq \vert f'(z)\vert$$ holds in  $\vert z \vert \leq r_{\psi}$,
	where $r_{\psi}\in (0,1)$ is the smallest root of the equation
	\begin{equation*}\label{r0-s}
	(1-r^2)m(r)-2r=0.
	\end{equation*}
	The radius constant $r_\psi$ is sharp.
\end{theorem}

Now we have our result for some well-known choices of the function $\psi$ and some recently studied:
\begin{corollary}\label{all-result}
	Let $g\in \mathcal{A}$. If $g <<f$, where $f\in \mathcal{S}_{c}^{*}(\psi)$ in $\mathbb{D}$ then
	$$\vert g'(z)\vert \leq \vert f'(z) \vert $$ holds in $\vert z \vert \leq r_{\psi},$
	where  $r_{\psi} \in (0,1)$ is the smallest root of 
	$$q(r)=0,$$ where
	\begin{itemize}
		\item [$(i)$]  if $\psi(z)= \frac{1+Dz}{1+Ez}$, where $-1\leq E<D\leq1$, then $$q(r)=(1-r^2)((1-Dr)/(1-Er))-2r.$$
		
		\item [$(ii)$] if $\psi(z)=\frac{1+(1-2\alpha)z}{1-z}$, where $0\leq\alpha<1$, then	$$q(r)=(1-r)(1-(1-2\alpha)r)-2r.$$
		
		\item [$(iii)$] if  $\psi(z)=\left(\frac{1+z}{1-z}\right)^{\eta}$, where $0<\eta \leq1$, then	 $$q(r)=(1-r^2)((1-r)/(1+r))^{\eta}-2r.$$
		
		\item [$(iv)$] if  $\psi(z)=\sqrt{2}-(\sqrt{2}-1)\sqrt{\frac{1-z}{1+2(\sqrt{2}-1)z}}$, then	$$q(r)=(1-r^2)\left(\sqrt{2}-(\sqrt{2}-1)\sqrt{\frac{1+r}{1-2(\sqrt{2}-1)r}}\right)-2r.$$
		
		\item [$(v)$] if  $\psi(z)=(b(1+z))^{1/a}$, where $a\geq1$ and $b\geq 1/2$, then $$q(r)=(1-r^2)(b(1-r))^{1/a}-2r.$$
		
		\item [$(vi)$]  if $\psi(z)=e^z$, then $$q(r)=(1-r^2)-2re^{r}.$$ 
		
		\item [$(vii)$] if $\psi(z)=z+\sqrt{1+z^2}$, then $$q(r)=(1-r^2)(\sqrt{1+r^2}-r)-2r.$$ 
		
		\item [$(viii)$] if  $\psi(z)=\frac{2}{1+e^{-z}}$, then $$q(r)=(1-r^2)-r(1+e^r).$$ 
		
		\item [$(ix)$]  if $\psi(z)=1+\sin{z}$, then $$q(r)=(1-r^2)(1-\sin{r})-2r.$$
	\end{itemize}
	The radii constants $r_{\psi}$ in context of the above cases are all sharp.
\end{corollary}

\section*{Conflict of interest}
	The authors declare that they have no conflict of interest.

\end{document}